\documentclass[11pt]{amsart}
\usepackage{amssymb,pb-diagram}
\usepackage{amsmath}
\usepackage{amsthm,amssymb,amsfonts}
\usepackage{amssymb,amscd}
\usepackage{epic,eepic,epsfig}
\usepackage{latexsym}
\usepackage{xy}
\xyoption{all}

\usepackage{amssymb,amsthm,enumerate,color,mathrsfs}
\def\en{\mathbb N}
\def\rn{\mathbb R^N}
\def\er{\mathbb R}

\def\dom{\operatorname{dom}}
\def\Lip{\operatorname{Lip}}

\newtheorem{Theorem}{Theorem}[section]

\newtheorem{Proposition}[Theorem]{Proposition}
\newtheorem{Lemma}[Theorem]{Lemma}
\newtheorem{Corollary}[Theorem]{Corollary}

\theoremstyle{definition}

\theoremstyle{definition}

\theoremstyle{remark}

\def \cal{\mathcal}

\def\Lip{\text{Lip}}

\def\eps{\varepsilon}

\def\Ndb{\mathbb N}

\def\Rdb{\mathbb R}

\def\eps{\varepsilon}

\def\Lip{\text{Lip}}

\newcommand{\bib}{\bibitem}

\begin{document}

\title{Approximation properties and Schauder decompositions in Lipschitz-free spaces}

\author{G. Lancien}

\address{Universit\'{e} de Franche-Comt\'{e}, Laboratoire de Math\'{e}matiques UMR 6623,
16 route de Gray, 25030 Besan\c{c}on Cedex, FRANCE.}
\email{gilles.lancien@univ-fcomte.fr}

\author{E. Perneck\'{a}*}
\address{*Department of Mathematical Analysis, Faculty of Mathematics and Physics, Charles University, Sokolovsk\'{a} 83, 186 75 Praha 8, Czech Republic.}
\address{*Universit\'{e} de Franche-Comt\'{e}, Laboratoire de Math\'{e}matiques UMR 6623,
16 route de Gray, 25030 Besan\c{c}on Cedex, FRANCE.}
\email{pernecka@karlin.mff.cuni.cz}



\subjclass[2010]{Primary 46B20; Secondary 46B80 }
\thanks{Both authors were partially supported by the P.H.C. Barrande 2012-26516YG. The second author was also supported by grants SVV-2012-265316, GA\v{C}R P201/11/0345 and RVO: 67985840.}

\maketitle

\begin{abstract}
We prove that the Lipschitz-free space over a doubling metric space has the bounded approximation
property. We also show that the Lipschitz-free spaces over $\ell_1^N$  or $\ell_1$ have monotone finite-dimensional Schauder decompositions.
\end{abstract}

\markboth{}{}

\section{Introduction}

 For $(M_1,d_1)$ and $(M_2,d_2)$ metric spaces and $f:M_1\to M_2$, we denote by $\Lip(f)$ the Lipschitz constant of $f$ given by $$\Lip(f)=\sup\left\{\frac{d_2(f(x),f(y))}{d_1(x,y)},\ x,y\in M_1,\ x\neq y\right\}.$$
Consider $(M,d)$ a \emph{pointed} metric space, i.e. a metric space equipped with a
distinguished element (origin) denoted $0$. Then, the space $\Lip_0(M)$ of all real-valued
Lipschitz functions $f$ on $M$ which satisfy $f(0)=0$, endowed with the norm
$$\|f\|_{\text{Lip}_0(M)}=\Lip(f)$$
is a Banach space. The Dirac map $\delta : M \to \text{Lip}_0(M)^*$ defined by $\langle g, \delta(p)\rangle =g(p)$
for $g\in \text{Lip}_0(M)$ and $p\in M$ is an isometric
embedding from $M$ into $\text{Lip}_0(M)^*$. The closed linear span of $\{\delta(p),\ p\in M\}$ is denoted
$\cal F (M)$ and called the \emph{Lipschitz-free space over} $M$ (or free space in short).
It follows from the compactness of the unit ball of $\text{Lip}_0(M)$ with respect to the topology of
pointwise convergence, that $\cal F (M)$ can be seen as the canonical predual of $\text{Lip}_0(M)$.
Then the weak$^*$-topology induced by $\cal F (M)$ on $\text{Lip}_0(M)$ coincides with the topology of
pointwise convergence on the bounded subsets of $\text{Lip}_0(M)$. Lipschitz-free spaces are a very
useful tool for abstractly linearizing Lipschitz maps. Indeed, if we identify through the Dirac map a
metric space $M$ with a subset of $\cal F (M)$, then any Lipschitz map from the metric space $M$ to a metric space
$N$ extends to a continuous linear map from $\cal F (M)$ to $\cal F (N)$ with the same Lipschitz constant (see \cite{W}
or Lemma 2.2 in \cite{GK}).
A comprehensive reference for the basic theory of the spaces of Lipschitz functions and their preduals, which
are called Arens-Eells spaces there, is the book \cite{W} by Weaver.

Despite the simplicity of their definition, very little is known about the linear structure of Lipschitz-free spaces over separable metric spaces. It is easy to see that $\cal F (\Rdb)$ is isometric to $L_1$.
However, adapting a theorem of Kislyakov \cite{Ki}, Naor and Schechtman proved in \cite{NS} that $\cal F (\Rdb^2)$ is not isomorphic to any subspace
of $L_1$. Then the metric spaces whose Lipschitz-free space is isometric to a subspace of $L_1$ have been
characterized by Godard in \cite{Gd}.

The aim of this paper is to study metric spaces $M$ such that $\cal F (M)$ has the bounded approximation
property (BAP) or admits a finite-dimensional Schauder decomposition (FDD). This kind of study was initiated
in the fundamental paper by Godefroy and Kalton \cite{GK}, where they proved that a Banach space $X$ has the
$\lambda$-BAP if and only if $\cal F (X)$ has the $\lambda$-BAP. In particular, for any finite dimensional
Banach space $E$, $\cal F (E)$ has the metric approximation property (MAP). Another major result from \cite{GK}
is that any separable Banach space has the so-called isometric lifting property. Refining the techniques used in
the proof of this result, Godefroy and Ozawa have proved in their recent work \cite{GO} that any separable Banach
space failing the BAP contains a compact subset whose Lipschitz-free space also fails the BAP. It is then natural,
as it is suggested in \cite{GO}, to try to describe the metric spaces whose Lipschitz-free space has BAP.
We address this question in Section 2. Our main result of this section (Corollary \ref{BAP for F(M)}) is that for
any doubling metric space $M$, the Lipschitz-free space $\cal F (M)$ has the BAP.

Then we try to find the Banach spaces such that the corresponding Lipschitz-free spaces have stronger approximation properties. The first result in this direction is due to Borel-Mathurin \cite{BoM}, who proved that $\cal F (\Rdb^N)$ admits a finite-dimensional Schauder decomposition. The decomposition constant obtained in \cite{BoM} depends on the dimension $N$. In Section 3 we show that $\cal F (\ell_1^N)$ and $\cal F(\ell_1)$ admit a monotone finite-dimensional Schauder decomposition. For that purpose, we use a particular technique for interpolating Lipschitz functions on hypercubes of $\Rdb^N$.

\section{Bounded approximation property for Lipschitz-free spaces and gentle partitions of unity}
\label{sec BAP for F(M)}

We first recall the definition of the bounded approximation property.

Let $1\leq\lambda<\infty$. A Banach space $X$ has the \emph{$\lambda$-bounded approximation property}
($\lambda$-BAP) if, for every $\varepsilon>0$ and every compact set $K\subset X$, there is a bounded finite-rank
linear operator $T:X\to X$ with $\|T\|\leq\lambda$ and such that  $\|T(x)-x\|\leq\varepsilon$ whenever $x\in K$. We say that
$X$ has the BAP if it has the $\lambda-$BAP for some $1\leq\lambda<\infty$.

Obviously, if there is a sequence of uniformly bounded finite-rank linear operators on a Banach space $X$ converging
in the strong operator topology to the identity on $X$, then $X$ has the BAP. For further information on the
approximation properties of Banach spaces we refer the reader to \cite{LT} or \cite{FHHMZ}.

\medskip We now detail a construction due to Lee and Naor \cite{LN} that we shall use. Let $(Y,d)$ be a metric space,
$X$ a closed subset of $Y$, $(\Omega, \Sigma,\mu)$ a measure space and $K>0$. Following \cite{LN} we say that a
function $\psi: \Omega \times Y \to [0,\infty)$ is a \emph{$K$-gentle partition of unity of $Y$ with respect to $X$}
if it satisfies the following:

\noindent (i) For all $x\in Y \setminus X$, the function $\psi_x:\omega \mapsto \psi(\omega,x)$ is in
$L_1(\mu)$ and $\|\psi_x\|_{L_1(\mu)}=1$.

\noindent (ii) For all $\omega \in \Omega$ and all $x$ in $X$, $\psi(\omega,x)=0$.

\noindent (iii) There exists a Borel measurable function $\gamma:\Omega \to X$ such that for all $x,y \in Y$
$$\int_\Omega |\psi(\omega,x)-\psi(\omega,y)|d(\gamma(\omega),x) \,d\mu(\omega) \le Kd(x,y).$$
Then, for $Y$ having a $K$-gentle partition of unity with respect to a separable subset $X$ and for $f$ Lipschitz on $X$, Lee and Naor define $E(f)$ by $E(f)(x)=f(x)$ if $x\in X$ and
$$E(f)(x)=\int_\Omega f(\gamma(\omega))\psi(\omega,x)\,d\mu(\omega)\ \ \ {\rm if}\ x\in Y\setminus X$$
and show that $\Lip(E(f)) \le 3K\Lip(f)$ (\cite{LN} Lemma 2.1).

\medskip Our general result is then the following.

\begin{Theorem}\label{BAPgentle} Let $(M,d)$ be a pointed separable metric space such that there exists a constant $K>0$ so that for any closed subset $X$ of $M$, $M$ admits a $K$-gentle partition of unity with respect to $X$. Then $\cal F (M)$ has the $3K$-BAP.
\end{Theorem}

\begin{proof} Our objective is to find a sequence of finite-rank linear operators on $\mathcal{F}(M)$ with norms bounded by $3K$ and
converging to the identity on $\mathcal{F}(M)$ in the strong operator topology. To this end, we first construct
a sequence of operators of appropriate qualities on the dual space $\Lip_0(M)$ so that they are adjoint operators and then pass
to $\mathcal{F}(M)$. To be more precise, we build  a sequence $(S_n)_{n=1}^\infty$ of $3K$-bounded finite-rank linear
operators on $\Lip_0(M)$ that are pointwise continuous on bounded subsets of $\Lip_0(M)$ and such that for all
$f\in \Lip_0(M)$, $(S_n(f))_{n=1}^\infty$ converges pointwise to $f$. This will imply that $S_n=T_n^*$, where $T_n$ is a finite-rank operator on $\cal F (M)$ such that $(T_n)_{n=1}^\infty$ is converging to the identity for the weak operator topology on $\cal F(M)$. Then the separability of $\cal F (M)$, Mazur's Lemma and a standard diagonal argument will yield the existence of a bounded sequence of finite-rank operators converging to the identity for the strong operator topology on $\cal F (M)$. Note that the operators obtained in this last step are made of convex combinations of the $T_n$'s. This preserves our control on their norms.

So, let $(x_n)_{n=1}^\infty$ be a dense sequence in $M$ and $0$ be the origin of $M$. Put $X_n=\{0,x_1,..,x_n\}$. For $f\in \Lip_0(M)$ we denote $R_n(f)$ the restriction of $f$ to $X_n$. The operator $R_n$, defined from $\Lip_0(M)$ to
$\Lip_0(X_n)$, is clearly of finite rank, pointwise continuous and such that $\|R_n\|\le
1$.

Thanks to our assumption that $M$ admits a $K$-gentle partition of unity with respect to $X_n$, we can apply Lee and Naor's construction to obtain an extension operator $E_n$ from $\Lip_0(X_n)$ to $\Lip_0(M)$. Note that it follows immediately from the definition of $E_n$ and Lebesgue's dominated convergence theorem that $E_n$ is pointwise continuous on bounded subsets of $\Lip_0(X_n)$.

Finally, we set $S_n=E_nR_n$. The sequence $(S_n)_{n=1}^\infty$ is
indeed a sequence of  finite-rank linear operators from $\Lip_0(M)$
to $\Lip_0(M)$ that are pointwise continuous on bounded subsets of $\Lip_0(M)$ and so that
$\|S_n\|\le 3K$ for all $n\in \Ndb$. To finish the proof, we
only need to show that for any $f\in \Lip_0(M)$, the sequence $(S_n(f))_{n=1}^\infty$ converges
pointwise to $f$. So let us fix $x\in M$, $f\in \Lip_0(M)$ and $\eps>0$. Let $n_0\in \Ndb$ such that $d(x,x_{n_0})\le \eps$. Then, for any $n\ge n_0$,
$$|f(x)-f(x_{n_0})|\le \eps \|f\|_{\Lip_0(M)}\ \ {\rm and}\ \ |S_n(f)(x)-f(x_{n_0})|\le 3K\eps\|f\|_{\Lip_0(M)}.$$
Therefore
$$|S_n(f)(x)-f(x)|\le (1+3K)\eps\|f\|_{\Lip_0(M)}.$$
This concludes our proof.

\end{proof}

We recall that a metric space $(M,d)$ is called \emph{doubling} if there exists a constant $D(M)>0$
 such that any open ball $B(p,R)$ in $M$ can be covered with at most $D(M)$ open balls of radius $\frac R2$. We can now state the main application of Theorem \ref{BAPgentle}.

\begin{Corollary}
\label{BAP for F(M)}
Let $(M,d)$ be a pointed doubling metric space. Then the Lipschitz-free space $\mathcal{F}(M)$ has the
$C(1+\log(D(M)))$-BAP, where $C$ is a universal constant.
\end{Corollary}
\begin{proof} We combine some of the important results from \cite{LN}. Namely, it follows from Lemma 3.8, Corollary 3.12
and Theorem 4.1 in \cite{LN} that if $M$ is a doubling metric space and $X$ is a closed subset of $M$,
then $M$ admits a $K (1+\log (D(M)))$-gentle partition of unity with respect to $X$ (where $K$ is a universal constant).
The conclusion is now a direct application of Theorem \ref{BAPgentle}
\end{proof}

\noindent {\bf Remarks.}

1) Let us mention that an extension operator with these properties
could also be obtained from a later construction
due to A. and Y. Brudnyi in \cite{Br}, where they use the notion of metric space of homogeneous type.  A Borel measure $\mu$ on a metric space $(M,d)$
is called \emph{doubling} if the measure of every open ball
is strictly positive and finite and if there is a constant $\delta(\mu)>0$ such that
$\mu(B(p,2R))\leq\delta(\mu)\mu(B(p,R))$ for all $p\in M$ and $R>0$. A
metric space endowed with a doubling measure is said to be of \emph{homogeneous type}. Clearly, a space of homogeneous type is doubling. Conversely, Luukkainen and Saksman proved in \cite{LS} that every complete doubling metric space $(M,d)$
carries a doubling measure $\mu$ such that $\delta(\mu)\leq c(D(M))$, where $c(D(M))$ is a constant
depending only on $D(M)$. More on doubling metric spaces and spaces of homogeneous type can be found
in \cite{Se} and \cite{H}.

2) We refer the reader to Lee and Naor's paper \cite{LN} for other examples of metric spaces admitting $K$-gentle partitions of unity such as negatively curved manifolds, special graphs or surfaces of bounded genus.

\medskip Let us conclude this section with a few words on the Lipschitz-free spaces over subsets of $\Rdb^N$. It is easily checked that for $N\in\en$, the space $\rn$ with the Euclidean norm is a
doubling metric space with doubling constant bounded above by $K^N$, where $K$ is a universal constant. This property is inherited by metric subspaces. So, it follows from Corollary \ref{BAP for F(M)} that for any closed subset $F$ of the Euclidean space $\Rdb^N$, $\cal F (F)$ has the $CN$-BAP for some universal constant $C$. It turns out that a better result can be derived from \cite{LN} and \cite {GK}.

\begin{Proposition}
\label{F(F) is complemented in F(RN)}
Let $N\in\en$ and $F$ be a closed subset of $\er^N$ equipped with the Euclidean norm. Then the Lipschitz-free space $\mathcal{F}(F)$ is isometric to a $C\sqrt{N}$-complemented subspace of the Lipschitz-free space
$\mathcal{F}(\er^N)$. In particular, $\cal F (F)$ has the $C\sqrt{N}$-BAP.
\end{Proposition}
\begin{proof}
We may assume, after translating $F$, that $0\in F$. The restriction to $F$ defined from
$\Lip_0(\Rdb^N)$ to $\Lip_0(F)$ is the adjoint operator of an isometry $J$ from $\cal F (F)$
into $\cal F (\Rdb^N)$. We can now apply a more precise result on extensions of Lipschitz functions coming from \cite{LN}. Indeed, it follows from Lemma 3.16 and Theorem 4.1 in \cite{LN} that $\Rdb^N$ equipped with the Euclidean
norm admits a $K\sqrt{N}$-gentle partition of unity with respect to $F$, where $K$ is a universal constant. So, there exists a weak$^*$ to weak$^*$ continuous linear operator $E:\Lip_0(F)\to\Lip_0(\er^N)$ such that
$E(f)\arrowvert_F=f$ for every $f\in\Lip_0(F)$ and  $\|E\|\le 3K\sqrt{N}$. Due to the weak$^*$-continuity of $E$, there exists a
bounded linear operator $P:\mathcal{F}(\er^N)\to\mathcal{F}(F)$ satisfying $P^*=E$. Moreover, thanks to the fact
that $E$ is an extension operator and by the Hahn-Banach theorem, $JP(\mu)=\mu$ for every $\mu\in J(\mathcal{F}(F))$.
Hence $JP$ is a linear projection from $\mathcal{F}(\er^N)$ onto $J(\mathcal{F}(F))$ such that $\|JP\|\le 3K\sqrt{N}$. This shows that $\mathcal{F}(F)$ is isometric to a $C\sqrt{N}$-complemented subspace of
$\mathcal{F}(\er^N)$, where $C$ is a universal constant. On the other hand, it is proved in \cite{GK} that $\cal F(\Rdb^N)$ has the MAP.
Therefore $\cal F(F)$ has the $C\sqrt{N}$-BAP.
\end{proof}

\section{Finite-dimensional Schauder decomposition of the Lipschitz-free space $\mathcal{F}(\ell_1)$}
\label{FDD of F(l_1)}

We recall the notion of finite-dimensional Schauder decomposition following the monograph of Lindenstrauss and Tzafriri \cite{LT}.

Let $X$ be a Banach space. A sequence $(X_n)_{n=1}^\infty$ of finite-dimensional subspaces of $X$ is called a \emph{finite-dimensional Schauder decomposition} of $X$ (FDD) if every $x\in X$ has a unique representation of the form $x=\sum_{n=1}^\infty x_n$ with $x_n\in X_n$ for every $n\in\en$.

If $(S_n)_{n=0}^\infty$, where $S_0\equiv0$, is a sequence of projections on $X$ satisfying $S_nS_m=S_{\min\{m,n\}}$ such that $0<\dim(S_{n}-S_{n-1})(X)<\infty$ and converging in the strong operator topology to the identity on $X$, then $(X_n)_{n=1}^\infty=\big((S_n-S_{n-1})(X)\big)_{n=1}^\infty$ is an FDD of $X$, for which the $S_n$'s are the partial sum projections. Then the sequence $(S_n)_{n=1}^\infty$ is bounded and $K=\sup_{n\in \Ndb} \|S_n\|$ is called the \emph{decomposition constant}. If $K=1$, then the decomposition is called \emph{monotone}.

\medskip For $N\in \Ndb$, the space $\Rdb^N$ equipped with the norm $\|x\|_{1}=\sum_{i=1}^N|x_i|$ is denoted $\ell_1^N$. The space $\left\{x=(x_i)_{i=1}^\infty \in \Rdb^\Ndb,\ \sum_{i=1}^\infty|x_i|<\infty\right\}$ equipped with the norm $\|x\|_{1}=\sum_{i=1}^\infty|x_i|$ is denoted $\ell_1$. We write $\mathbf{0^N}$ for the origin in $\ell_1^N$ and $\mathbf{0}$ for the origin in $\ell_1$. Our result is the following.

\begin{Theorem}\label{fdd}
The Lipschitz-free spaces $\mathcal{F}(\ell_1)$ and $\mathcal{F}(\ell_1^N)$ admit
monotone finite-dimensional Schauder decompositions.
\end{Theorem}

Let $X$ be $\ell_1$ or $\ell_1^N$. It follows from the classical theory that we only need to build a sequence of contractive finite-rank linear projections $(S_n)_{n=1}^\infty$ on $\cal F(X)$ such that
$S_nS_m=S_{\min\{m,n\}}$ for all $m,n\in\en$ and that $\overline{\bigcup\limits_{n=1}^{\infty}S_n(\cal F(X))}=\cal F(X)$.
As in the proof of Theorem \ref{BAPgentle} we shall work on the dual space and construct a sequence of contractive
weak$^*$ to weak$^*$ continuous finite-rank linear projections on $\Lip_0(X)$,
possessing the commuting property and converging to the identity on $\Lip_0(X)$ in
the weak$^*$-operator topology. The general idea will be to take an increasing sequence of closed
bounded subsets of $X$ and associate with each of these sets a finite-rank linear operator on
$\Lip_0(X)$ so that the image of a function under this operator has values close to the values of the
original function at the points of the considered closed set. However, unlike the situation in our previous section, we have the linear structure of the metric space $X$ at our disposal. This will
enable us to work accurately enough to obtain a monotone FDD for $\mathcal{F}(X)$.

\subsection{Notation and interpolation Lemma}
\label{Notation}

Put $\en_0=\en\cup\{0\}$ and fix $N\in\en$. We denote by $C\left(y,R\right)$, where $y\in\rn$ and $0<R<\infty$, the hypercube
$$C(y,R)=\left\{x\in\rn,\ \sup_{1\le i\le N}|x_i-y_i|\leq\frac R2\right\}.$$
For $y\in\rn$, $0<R<\infty$ and $\delta\in\{-1,1\}^N$, the symbol $A_\delta(y,R)$ stands for the vertex $y+\frac R2\delta$ of the hypercube $C\left(y,R\right)$.

The following interpolation on $C(y,R)$ of a function defined on its vertices will be the crucial tool for our proof. Let $y\in\rn$, $0<R<\infty$, $x\in C\left(y,R\right)$ and let $f:\dom(f)\to\er$ satisfy $\left\{A_\delta\left(y,R\right),\ \delta\in\{-1,1\}^N\right\}\subset \dom(f)\subset\rn$. We define inductively:
\begin{align*}
\Lambda_{\gamma}\left(f,\,C\left(y,R\right)\right)(x)=&\frac{x_1-y_1+\frac R2}{R}f\left(A_{(1,\gamma_1,\dots,\gamma_{N-1})}\left(y,R\right)\right)\\
&+\left(1-\frac{x_1-y_1+\frac R2}{R}\right)f\left(A_{(-1,\gamma_1,\dots,\gamma_{N-1})}\left(y,R\right)\right)
\end{align*}
for each $\gamma=(\gamma_1,\dots,\gamma_{N-1})\in\{-1,1\}^{N-1}$,
\begin{align*}
\Lambda_{\gamma}\left(f,\,C\left(y,R\right)\right)(x)=&\frac{x_j-y_j+\frac R2}{R}\Lambda_{(1,\gamma_1,\dots,\gamma_{N-j})}\left(f,\,C\left(y,R\right)\right)(x)\\
&+\left(1-\frac{x_j-y_j+\frac R2}{R}\right)\Lambda_{(-1,\gamma_1,\dots,\gamma_{N-j})}\left(f,\,C\left(y,R\right)\right)(x)
\end{align*}
for each $j\in\{2,\dots,N-1\}$ and $\gamma=(\gamma_1,\dots,\gamma_{N-j})\in\{-1,1\}^{N-j}$,
and
\begin{align}
\label{Lambda}
\Lambda\left(f,\,C\left(y,R\right)\right)(x)=&\frac{x_N-y_N+\frac R2}{R}\Lambda_{(1)}\left(f,\,C\left(y,R\right)\right)(x)\\\nonumber
&+\left(1-\frac{x_N-y_N+\frac R2}{R}\right)\Lambda_{(-1)}\left(f,\,C\left(y,R\right)\right)(x).\nonumber
\end{align}
Let us use the following convention: $\{-1,1\}^0:=\{\emptyset\}$ and $\Lambda_\emptyset=\Lambda$.

Let $I_1,\dots,I_N$ be closed intervals in $\er$. We shall say that a function $\Phi: I_1\times\dots\times I_N\to\er$ has the property (AF) on $I_1\times\dots\times I_N\subset\rn$ if its restriction to any segment lying in $I_1\times\dots\times I_N$ and parallel to one of the coordinate axes is affine. A function having the property (AF) on $I_1\times\dots\times I_N$ is uniquely determined by its values at the vertices of $I_1\times\dots\times I_N$. Observe that $\Lambda\left(f,\,C\left(y,R\right)\right)$ has the property (AF) on $C(y,R)$ and coincides with the function $f$ at the vertices of $C\left(y,R\right)$.

We now state and prove our basic interpolation lemma.

\begin{Lemma}\label{Lip constant of Lambda} Let $y\in\rn$, $0<R<\infty$ and let $f:\dom(f)\to\er$ be a function satisfying  $\left\{A_\delta\left(y,R\right),\ \delta\in\{-1,1\}^N\right\}\subset \dom(f)\subset\rn$. Recall that $\Rdb^N$ is equipped with the $\ell_1$-norm. Then
$$
{\rm Lip}\left(\Lambda\left(f,\,C\left(y,R\right)\right)\right)= {\rm Lip}\left(f|_{\left\{A_\delta\left(y,R\right),\ \delta\in\{-1,1\}^N\right\}}\right).
$$
\end{Lemma}

\begin{proof}  It follows clearly from its definition that  $\Lambda\left(f,C\left(y,R\right)\right)$ is differentiable in the interior of $C\left(y,R\right)$. We shall prove that for any $1\le i\le N$ and any $x$ in the interior of $C\left(y,R\right)$

\begin{equation*}
\left|\frac{\partial\Lambda\left(f,\,C\left(y,R\right)\right)}{\partial x_i}(x)\right|\leq \Lip\left(f|_{\left\{A_\delta\left(y,R\right),\ \delta\in\{-1,1\}^N\right\}}\right).
\end{equation*}
Since $\Rdb^N$ is equipped with $\|\ \|_1$, the conclusion of our lemma will then follow directly from the mean value theorem.

So let $x$ be an interior point of $C\left(y,R\right)$, that is $x$ so that $y_i-\frac{R}{2}<x_i<y_i+\frac{R}{2}$ for all $1\le i\le N$. Consider first $\gamma,\tilde{\gamma}\in\{-1,1\}^{N-1}$ such that there exists a unique $k\in\{1,\dots,N-1\}$ satisfying $\gamma_k\neq\tilde{\gamma}_k$. Then
\begin{align*}
\big|&\Lambda_\gamma\left(f,\,C\left(y,R\right)\right)(x)-\Lambda_{\tilde{\gamma}}\left(f,\,C\left(y,R\right)\right)(x)\big|\\
&=\Bigg|\left(1-\frac{x_1-y_1+\frac {R}{2}}{R}\right)\left(f\left(A_{(-1,\gamma_1,\dots,\gamma_{N-1})}\left(y,R\right)\right)-f\left(A_{(-1,\tilde{\gamma}_1,\dots,\tilde{\gamma}_{N-1})}\left(y,R\right)\right)\right)\\
&\phantom{=\Bigg|}+\frac{x_1-y_1+\frac{R}{2}}{R}\left(f\left(A_{(1,\gamma_1,\dots,\gamma_{N-1})}\left(y,R\right)\right)-f\left(A_{(1,\tilde{\gamma}_1,\dots,\tilde{\gamma}_{N-1})}\left(y,R\right)\right)\right)\Bigg|\\
&\leq R\,\Lip\left(f|_{\left\{A_\delta\left(y,R\right),\ \delta\in\{-1,1\}^N\right\}}\right).
\end{align*}
Further, one shows by induction on $j\in\{1,\dots,N-1\}$ that for all  $\gamma,\tilde{\gamma}\in\{-1,1\}^{N-j}$ such that there exists a unique $k\in\{1,\dots,N-j\}$ satisfying $\gamma_k\neq\tilde{\gamma}_k$ we have
\begin{align}
\label{Lambda-Lambda}
\big|\Lambda_\gamma\left(f,\,C\left(y,R\right)\right)&(x)-\Lambda_{\tilde{\gamma}}\left(f,\,C\left(y,R\right)\right)(x)\big|\\\nonumber
&=\Bigg|\left(1-\frac{x_j-y_j+\frac {R}{2}}{R}\right)\Big(\Lambda_{(-1,\gamma_1,\dots,\gamma_{N-j})}\left(f,\,C\left(y,R\right)\right)(x)\\\nonumber
&\phantom{=\Bigg|\left(1-\frac{x_j-y_j+\frac {R}{2}}{R}\right)\Big(}-\Lambda_{(-1,\tilde\gamma_1,\dots,\tilde\gamma_{N-j})}\left(f,\,C\left(y,R\right)\right)(x)\Big)\\\nonumber
&\phantom{=\Bigg|}+\frac{x_j-y_j+\frac {R}{2}}{R}\Big(\Lambda_{(1,\gamma_1,\dots,\gamma_{N-j})}\left(f,\,C\left(y,R\right)\right)(x)\\\nonumber
&\phantom{=\Bigg|+\frac{x_j-y_j+\frac {R}{2}}{R}\Big(}-\Lambda_{(1,\tilde\gamma_1,\dots,\tilde\gamma_{N-j})}\left(f,\,C\left(y,R\right)\right)(x)\Big)\Bigg|\\\nonumber
&\leq R\,\Lip\left(f|_{\left\{A_\delta\left(y,R\right),\ \delta\in\{-1,1\}^N\right\}}\right).\nonumber
\end{align}
Now, for $\gamma\in\{-1,1\}^{N-1}$ and $i\in\{1,\dots,N\}$,
\begin{align*}
\left|\frac{\partial\Lambda_\gamma\left(f,\,C\left(y,R\right)\right)}{\partial x_i}(x)\right|&=\left\{\begin{array}{ll}
\bigg|\frac{f\left(A_{(1,\gamma_1,\dots,\gamma_{N-1})}\left(y,R\right)\right)}{R}\\
\phantom{\bigg|}-\frac{f\left(A_{(-1,\gamma_1,\dots,\gamma_{N-1})}\left(y,R\right)\right)}{R}\bigg|&\textup{if }\,i=1,\\
0&\textup{if }\,i>1.
\end{array}
\right.
\end{align*}
Therefore
$$
\left|\frac{\partial\Lambda_\gamma\left(f,\,C\left(y,R\right)\right)}{\partial x_i}(x)\right|\leq\Lip\left(f|_{\left\{A_\delta\left(y,R\right),\ \delta\in\{-1,1\}^N\right\}}\right).
$$
Further, for $j\in\{2,\dots,N\}$, $\gamma \in\{-1,1\}^{N-j}$ and $i\in\{1,\dots,N\}$,
\begin{align*}
\left|\frac{\partial\Lambda_\gamma\left(f,\,C\left(y,R\right)\right)}{\partial x_i}(x)\right|&=\left\{\begin{array}{ll}
\bigg|\frac{R-\left(x_j-y_j+\frac {R}{2}\right)}{R}\frac{\partial\Lambda_{(-1,\gamma_1,\dots,\gamma_{N-j})}\left(f,\,C\left(y,R\right)\right)}{\partial x_i}(x)\\
\phantom{\bigg|}+\frac{x_j-y_j+\frac {R}{2}}{R}\frac{\partial\Lambda_{(1,\gamma_1,\dots,\gamma_{N-j})}\left(f,\,C\left(y,R\right)\right)}{\partial x_i}(x)\bigg|&\textup{if }i<j\\
\bigg|\frac{\Lambda_{(1,\gamma_1,\dots,\gamma_{N-j})}\left(f,\,C\left(y,R\right)\right)(x)}{R}\\
\phantom{\bigg|}-\frac{\Lambda_{(-1,\gamma_1,\dots,\gamma_{N-j})}\left(f,\,C\left(y,R\right)\right)(x)}{R}\bigg|&\textup{if }i=j\\
0&\textup{if }i>j.
\end{array}
\right.
\end{align*}
Consequently, using (\ref{Lambda-Lambda}) and an induction on $j$, one gets that for all $j\in\{1,\dots,N\}$, $i\in\{1,\dots,N\}$ and $\gamma \in\{-1,1\}^{N-j}$,
$$
\left|\frac{\partial\Lambda_\gamma\left(f,\,C\left(y,R\right)\right)}{\partial x_i}(x)\right|\leq\Lip\left(f|_{\left\{A_\delta\left(y,R\right),\
\delta\in\{-1,1\}^N\right\}}\right).
$$
This concludes the proof.

\end{proof}

We now finish setting our notation. Provided that $\varepsilon=(\varepsilon_1,\dots,\varepsilon_N)\in\{-1,1\}^N$, $y\in\rn$, $h=(h_1,\dots,h_N)\in\en^N_0$ and $k\in\en_0$, we denote
$$x^{\varepsilon,y}_{h,k}=y+2^{-k-1}\varepsilon+2^{-k}(\varepsilon_1h_1,\dots,\varepsilon_N h_N).$$
Next, if $0<R<\infty$ and $t\in \Rdb$, we define $\pi_R(t)$ to be the nearest point to $t$ in $[-\frac{R}{2},\frac{R}{2}]$. Then we define $\Pi_R^N(x)=(\pi_R(x_1),\dots,\pi_R(x_N))$ for all $x\in \Rdb^N$.
It is easily checked that $\Pi_R^N$ is a retraction from $\ell_1^N$ onto $C\left(\mathbf{0^N},R\right)$ and
that $\Lip(\Pi_R^N)=1$. In fact, $\Pi_R^N$ is the nearest point mapping to $C\left(\mathbf{0^N},R\right)$ and is 1-Lipschitz
in both $\|\ \|_1$ and $\|\ \|_2$ on $\Rdb^N$.

\medskip Finally,  we define $\rho_N$ to be the canonical projection from $\ell_1$ onto $\ell_1^N$ and $\tau_N$ to be the canonical injection from $\ell_1^N$ into $\ell_1$. Namely  $\rho_N(x)=(x_1,\dots,x_N)$ for any $x=(x_i)_{i=1}^\infty\in\ell_1$ and  $\tau_N(x)=(x_1,\dots,x_N,0,\dots)$ for every $x=(x_1,\dots,x_N)\in \ell_1^N$.

\subsection{Proof of Theorem \ref{fdd}}

We detail the argument for $\cal F(\ell_1)$. As we have announced in the note below
the formulation of Theorem \ref{fdd}, we perform first a construction of projections having the desired qualities on $\Lip_0(\ell_1)$.

So, for $f\in\Lip_0(\ell_1)$, $n\in\en$ and $x\in \ell_1$ we define

$$Q_n(f)(x)=P_n(f\circ \tau_n)(\rho_n(x)),$$
with
$$P_n(g)(u)=\Lambda\left(g,
C\left(x^{\varepsilon,\mathbf{0^n}}_{h,n-1},2^{1-n}\right)\right)(\Pi^n_{2^n}(u)),\ \ {\rm for}\ g\in \Lip_0(\ell_1^n)\ {\rm and}\ u\in \ell_1^n,$$
where $\varepsilon\in\{-1,1\}^n$ and $h\in\{0,\dots,2^{2n-2}-1\}^n$ are chosen such that $\Pi^n_{2^n}(u)\in C\left(x^{\varepsilon,\mathbf{0^n}}_{h,n-1},2^{1-n}\right)$.

\noindent Let us mention that in the above construction the symbols $x^{\varepsilon,\mathbf{0^n}}_{h,n-1}$, $C\left(x^{\varepsilon,\mathbf{0^n}}_{h,n-1},2^{1-n}\right)$ and $\Lambda\left(g,
C\left(x^{\varepsilon,\mathbf{0^n}}_{h,n-1},2^{1-n}\right)\right)$ are meant in $\Rdb^n$, or acting on $\Rdb^n$. In the sequel, the information on the dimension considered for hypercubes or for points $x^{\varepsilon,y}_{h,k}$ shall be carried by the centre of a hypercube or by $y$ respectively, which most of the time will be $\mathbf{0^n}$. Finally, we denote $V_n$ the set of all vertices of all cubes $C\left(x^{\varepsilon,\mathbf{0^n}}_{h,n-1},2^{1-n}\right)$ for $\varepsilon\in\{-1,1\}^n$ and $h\in\{0,\dots,2^{2n-2}-1\}^n$.

Before we proceed with the proof, let us describe the operator $Q_n$. The hypercube $C(\mathbf{0^n},2^n)$ of
$\ell_1^n$ is split into small hypercubes of edge length equal to $2^{1-n}$. If $x$ belongs to one of the small
hypercubes, then $Q_n(f)(x)$ is the value obtained by performing the interpolation $\Lambda$ for the restriction
of $f$ to the vertices of this small hypercube. If $x$ does not belong to $C(\mathbf{0^n},2^n)$, then $Q_n(f)(x)$
is defined to be $Q_n(f)(r_n(x))$, where $r_n=\Pi_{2^n}^n\circ \rho_n$ is the natural retraction from $\ell_1$
onto $C(\mathbf{0^n},2^n)$. In rough words, let us say that as we go from step $n$ to
step $n+1$, we perform the three following operations: we add one dimension to our
hypercubes, we double the edge length of the large hypercube and divide by two the edge length of the
small hypercubes.

\medskip We now make a simple observation.

\begin{Lemma}\label{af} Let $m>n$ in $\Ndb$. Assume that $g\in {\rm Lip}_0(\ell_1^n)$. Then the function
$P_n(g)$ has the property (AF) on each hypercube $C\left(x^{\varepsilon,\mathbf{0^n}}_{h,m-1},2^{1-m}\right)$
where $\eps \in \{-1,1\}^n$ and $h \in \Ndb_0^n$ (note here that these hypercubes are considered in $\Rdb^n)$.
\end{Lemma}

\begin{proof} The assertion is clear if the hypercube $C\left(x^{\varepsilon,\mathbf{0^n}}_{h,m-1},2^{1-m}\right)$ lies inside $C\left(\mathbf{0^n},2^{n}\right)$. Assume now that it is not the case. First, it is easily checked that $\Pi_{2^n}^n$ has the property (AF) on $C\left(x^{\varepsilon,\mathbf{0^n}}_{h,m-1},2^{1-m}\right)$. Besides, the image by $\Pi_{2^n}^n$ of a segment in $C\left(x^{\varepsilon,\mathbf{0^n}}_{h,m-1},2^{1-m}\right)$ that is parallel to a coordinate axis is either a point or a segment parallel to a coordinate axis. Finally, $\Pi_{2^n}^n\left(C\left(x^{\varepsilon,\mathbf{0^n}}_{h,m-1},2^{1-m}\right)\right)$ is included in a face of one of the hypercubes in the tilling of $C\left(\mathbf{0^n},2^{n}\right)$. On this face $P_n(g)$ has the property (AF). The conclusion follows.

\end{proof}

Let us proceed with the proof of Theorem \ref{fdd}. Fix $n\in \Ndb$. First, it is clear that
$$
Q_n(f)(\mathbf{0})=f(\mathbf{0})=0.
$$
Then, using Lemma \ref{Lip constant of Lambda} and the fact that $1=\Lip(\tau_n)=\Lip(\rho_n)=\Lip(\Pi^n_{2^n})$,
we get that for all $x,y \in \ell_1$
\begin{align*}
|Q_n(f)(x)-Q_n(f)(y)|&\leq \|f\circ\tau_n\|_{\Lip_0(\ell_1)}\left\|\Pi^n_{2^n}(\rho_n(x))-\Pi^n_{2^n}(\rho_n(y))\right\|_1\\
&\leq \|f\|_{\Lip_0(\ell_1)}\left\|x-y\right\|_1.
\end{align*}
The map $f \mapsto \Lambda\left(f,C\left(y,R\right)\right)$ is clearly linear. Then, the linearity of $Q_n$ follows easily. Moreover, $Q_n(f)$ is uniquely determined by the values of $f$ at the elements of the finite set $V_n$. Hence $Q_n:\Lip_0(\ell_1)\to\Lip_0(\ell_1)$ is a well defined finite-rank linear operator with $\|Q_n\|\leq1$.

Consider now $m,n\in\en$ so that $m\leq n$. Then $Q_n(f)\circ\tau_m=f\circ\tau_m$ on $V_m$. Indeed, for $A=(A_1,\dots,A_m)\in V_m$, we have that $\rho_n(\tau_m(A))\in V_n$. So
\begin{align*}
Q_n(f)(\tau_m(A))=f(\tau_n(A_1,\dots,A_m,\underbrace{0,\dots,0}_{n-m}))=f(\tau_m(A)).
\end{align*}
Thus $Q_m(Q_n(f))=Q_m(f)$ on $\ell_1$ by definition.

Suppose now $m>n$ and assume that $j\in\{1,\dots,m\}$, $\lambda\in[0,1]$ and that $x,\tilde x\in C\left(x^{\varepsilon,\mathbf{0^m}}_{h,m-1},2^{1-m}\right)$, where $\varepsilon\in\{-1,1\}^m$ and $h\in\{0,\dots,2^{2m-2}-1\}^m$, are such that $x_i=\tilde x_i$ for $i\neq j$. Then
\begin{align*}
Q_n(f)(\tau_m(\lambda x+(1-\lambda)\tilde x))&=P_n(f\circ\tau_n)(\rho_n(\tau_m(\lambda x+(1-\lambda)\tilde x)))\\
&=P_n(f\circ\tau_n)(\lambda\rho_n(\tau_m(x))+(1-\lambda)\rho_n(\tau_m(\tilde x)))\\
&=\lambda P_n(f\circ\tau_n)(\rho_n(\tau_m(x)))\\
&\phantom{=}+(1-\lambda)P_n(f\circ\tau_n)(\rho_n(\tau_m(\tilde x)))\\
&=\lambda Q_n(f)(\tau_m(x))+(1-\lambda)Q_n(f)(\tau_m(\tilde x)).
\end{align*}
In the above we have used that $\rho_n$ and $\tau_m$ are affine, that
$$\rho_n\left(\tau_m\left(C\left(x^{\varepsilon,\mathbf{0^m}}_{h,m-1},2^{1-m}\right)\right)\right)
=C\left(x^{\left(\varepsilon_1,\dots,\varepsilon_n\right),\mathbf{0^n}}_{\left(h_1,\dots,h_n\right),
m-1},2^{1-m}\right)$$ and the fact that $P_n(f\circ \tau_n)$ has the property (AF) on  $C\left(x^{\left(\varepsilon_1,\dots,\varepsilon_n\right),\mathbf{0^n}}_{\left(h_1,\dots,h_n\right),
m-1},2^{1-m}\right)$ (see Lemma \ref{af}). So, $Q_n(f)\circ\tau_m$ has the property (AF) on each hypercube $C\left(x^{\varepsilon,\mathbf{0^m}}_{h,m-1},2^{1-m}\right)$, where $\varepsilon\in\{-1,1\}^m$ and $h\in\{0,\dots,2^{2m-2}-1\}^m$.

\noindent Therefore, we obtain that for all $x\in\ell_1$ and $f\in \Lip_0(\ell_1)$
\begin{align*}
Q_m(Q_n(f))(x)&=Q_n(f)\left(\tau_m\left(\Pi^m_{2^m}\left(\rho_m(x)\right)\right)\right)\\
&=P_n(f\circ\tau_n)\left(\rho_n\left(\tau_m\left(\Pi^m_{2^m}\left(\rho_m(x)\right)\right)\right)\right)\\
&=P_n(f\circ\tau_n)\left(\pi_{2^m}(x_1),\dots,\pi_{2^m}(x_n)\right)\\
&=P_n(f\circ\tau_n)\left(\Pi^n_{2^n}\left(\pi_{2^m}(x_1),\dots,\pi_{2^m}(x_n)\right)\right).
\end{align*}

\noindent We now use the fact that $\pi_{2n}\pi_{2m}=\pi_{2n}$ to get

\begin{align*}
Q_m(Q_n(f))(x)&=P_n(f\circ\tau_n)\left(\pi_{2^n}(x_1),\dots,\pi_{2^n}(x_n)\right)\\
&=P_n(f\circ\tau_n)\left(\Pi^n_{2^n}\left(\rho_n(x)\right)\right)=Q_n(f)(x).
\end{align*}
Hence the formula $Q_mQ_n=Q_{\min\{m,n\}}$ is also satisfied for $m>n$.

By construction, for each $n$ in $\en$, $Q_n$ is continuous on $\Lip_0(\ell_1)$ endowed with the topology of pointwise convergence, thus it is also weak$^*$ to weak$^*$ continuous on $\Lip_0(\ell_1)$.

Furthermore, $\big(Q_n(f)\big)_{n=1}^\infty$ converges pointwise to $f$ for every $f\in\Lip_0(\ell_1)$. Indeed, for  given $f\in\Lip_0(\ell_1)$, $x\in\ell_1$ and $\eta>0$, we can find $n_0\in\en$ such that for all $n\ge n_0$ $$\|f\|_{\Lip_0(\ell_1)}\sum_{i=n+1}^{\infty}|x_i|<\frac{\eta}{4},\ \ \rho_n(x)\in
C(\mathbf{0^n},2^n)\ \ {\rm and}\ \ n2^{1-n}\|f\|_{\Lip_0(\ell_1)}<\frac{\eta}{4}.$$
Thus, for any $n\geq n_0$, we get
\begin{align*}
|Q_n(f)(x)-f(x)|&\leq|Q_n(f)(x)-Q_n(f)(\tau_n(\rho_n(x)))|\\
&\phantom{\leq}+|Q_n(f)(\tau_n(\rho_n(x)))-f(\tau_n(A))|\\
&\phantom{\leq}+|f(\tau_n(A))-f(\tau_n(\rho_n(x)))|+|f(\tau_n(\rho_n(x)))-f(x)|,
\end{align*}
where $A\in\er^n$ is a vertex of a hypercube $C\left(x^{\varepsilon,\mathbf{0^n}}_{h,n-1},2^{1-n}\right)$, with $\varepsilon\in\{-1,1\}^n$ and $h\in\{0,\dots,2^{2n-2}-1\}^n$, containing $\rho_n(x)$.

\noindent Since $f$ and $Q_n(f)$ are $\|f\|_{\Lip_0(\ell_1)}$-Lipschitz and $\|\tau_n(A)-\tau_n(\rho_n(x))\|_{1} \le n2^{1-n}$, we deduce that

\begin{align*}
|Q_n(f)(x)-f(x)|&\leq 2\|f\|_{\Lip_0(\ell_1)}\left(\|\tau_n(\rho_n(x))-x\|_{1}+
\|\tau_n(A)-\tau_n(\rho_n(x))\|_{1}\right)\\
&\leq 2\|f\|_{\Lip_0(\ell_1)}\left(\sum_{i=n+1}^{\infty}{|x_i|}+n2^{1-n}\right)<\eta.
\end{align*}

Now, it follows from the weak$^*$ continuity of $Q_n$ that $Q_n=S_n^*$, where $(S_n)_{n=1}^\infty$ is a sequence of finite-rank bounded linear projections on $\mathcal{F}(\ell_1)$. The sequence $(S_n)_{n=1}^\infty$ satisfies that $\|S_n\|\leq 1$ for each $n\in\en$ and that $S_mS_n=S_{\min\{m,n\}}$ for every $m,n\in\en$.

The fact that $(Q_n)_{n=1}^\infty$ converges to the identity with respect to the weak$^*$ operator topology then implies that $(S_n(\mu))_{n=1}^\infty$ converges weakly to $\mu$ for every $\mu\in\mathcal{F}(\ell_1)$. Therefore $\overline{\bigcup\limits_{n=1}^{\infty}S_n(\mathcal{F}(\ell_1))}=
\mathcal{F}(\ell_1)$. In view of these properties, the sequence $(S_n)_{n=1}^\infty$ determines a monotone FDD of $\mathcal{F}(\ell_1)$. The proof of Theorem \ref{fdd} is now complete.

\medskip\noindent {\bf Remark.} The proof for $\ell_1^N$ is clearly simpler and the sequence $(Q_n)_{n=1}^\infty$ can be directly given by
$$Q_n(f)(x)=\Lambda\left(f,
C\left(x^{\varepsilon,\mathbf{0^N}}_{h,n-1},2^{1-n}\right)\right)\left(\Pi^N_{2^n}(x)\right),$$
where $\varepsilon\in\{-1,1\}^N$ and $h\in\{0,\dots,2^{2n-2}-1\}^N$ are such that $\Pi^N_{2^n}(x)\in C\left(x^{\varepsilon,\mathbf{0^N}}_{h,n-1},2^{1-n}\right)$.

\medskip\noindent {\bf Aknowledgements.} Both authors wish to thank G. Godefroy and P. H\'{a}jek for very useful discussions.

\end{document}